\documentclass[pdflatex,sn-mathphys-ay]{sn-jnl}

\usepackage{graphicx}
\usepackage{multirow}
\usepackage{amsmath,amssymb,amsfonts}
\usepackage{amsthm}
\usepackage{mathrsfs}
\usepackage[title]{appendix}
\usepackage{xcolor}
\usepackage{textcomp}
\usepackage{manyfoot}
\usepackage{booktabs}
\usepackage{algorithm}
\usepackage{algorithmicx}
\usepackage{algpseudocode}
\usepackage{listings}
\usepackage{mathtools}
\usepackage{bm}

\theoremstyle{thmstyleone}
\newtheorem{theorem}{Theorem}
\newtheorem{proposition}[theorem]{Proposition}

\theoremstyle{thmstyletwo}

\theoremstyle{thmstylethree}
\newtheorem{remark}[theorem]{Remark}

\newcommand{\E}{\mathbb{E}}
\newcommand{\Pp}{\mathbb{P}}

\newcommand{\ind}{\mathbf{1}}

\begin{document}

\title[Stochastic Epidemics and Heterogeneous SIR Equations]
{From Individual-Based Stochastic Epidemics to Heterogeneous SIR Equations}

% Replace the placeholder names, affiliations, and email addresses below.
\author[1]{\fnm{Olga} \sur{Izyumtseva}}
%\email{Olga.Iziumtseva1@nottingham.ac.uk}

\author[1]{\fnm{Wasiur R.} \sur{KhudaBukhsh}}
% \email{author.email@example.edu}

\author[2,3]{\fnm{M. Gabriela M.} \sur{Gomes}}
% \email{author.email@example.edu}

\author*[4]{\fnm{Grzegorz A.} \sur{Rempala}}
\email{rempala.3@osu.edu}

\affil[1]{\orgdiv{School of Mathematical Sciences},
\orgname{University of Nottingham},
\orgaddress{\city{Nottingham}, \country{United Kingdom}}}

\affil[2]{\orgdiv{Department of Mathematics and Statistics},
\orgname{University of Strathclyde},
\orgaddress{\city{Glasgow}, \country{United Kingdom}}}

\affil[3]{\orgdiv{Centre for Mathematics and Applications},
\orgname{NOVA School of Science and Technology},
\orgaddress{\city{Caparica}, \country{Portugal}}}

\affil*[4]{\orgdiv{Division of Biostatistics},
\orgname{The Ohio State University},
\orgaddress{\city{Columbus}, \state{OH}, \country{USA}}}

\abstract{
We develop a stochastic framework for a broad class of
heterogeneous SIR epidemic models. In the finite-population construction, each
initially susceptible individual is assigned a fixed nonnegative
susceptibility, and infection occurs when the accumulated population-level
infection pressure exceeds an individual random threshold. Infectious periods
are independent and exponentially distributed with a common recovery rate.
For any susceptibility distribution with finite mean, we prove a
uniform-on-compact law of large numbers for the susceptible, infectious,
removed, and cumulative infection-pressure processes. In the limit, an
individual with susceptibility \(\lambda\) remains susceptible under cumulative
pressure \(a\) with probability \(\exp(-a\lambda)\). It follows that the
susceptible fraction is given by the Laplace transform of the initial
susceptibility distribution, while the incidence rate is governed by the mean
susceptibility among those who remain susceptible.
The resulting limits recover several familiar heterogeneous SIR systems,
including the classical power-law model, and also yield other closed nonlinear
incidence forms. The framework therefore provides a unified probabilistic
foundation for deterministic epidemic models with persistent individual
heterogeneity.
}


\maketitle
\tableofcontents
\section{Introduction}\label{sec:intro}

Individual heterogeneity is one of the principal mechanisms by which real
epidemics depart from the classical homogeneous SIR model. Individuals may
differ persistently in biological susceptibility, prior immunity, contact
activity, exposure opportunities, behavior, or other traits that alter their
risk of infection. Such heterogeneity does more than change an average
transmission coefficient. Individuals with larger risk are preferentially
infected early, so the susceptible population is progressively enriched for
individuals with smaller risk. This selective-depletion mechanism \citep{VaupelYashin1985} can alter
epidemic growth, prevalence, final size, and herd-immunity thresholds
\citep{Dwyer1991,Novozhilov2008,Miller2014,
GomesEtAl2022,MontalbanEtAl2022}.

The  effects of heterogeneity  have been studied through two complementary mathematical approaches. The first
starts from deterministic compartmental equations with a continuously
distributed individual trait. The susceptible population is represented by a
density over trait space, and the aggregate epidemic dynamics are obtained by
integrating over that density.  For instance, nearly two decades ago, Novozhilov derived low-dimensional reductions and final-size relations for heterogeneous SIR models and showed that Gamma-distributed susceptibility gives rise to a power-law incidence function~\citep{Novozhilov2008}. During the COVID-19 pandemic, Gomes and collaborators further developed this perspective and emphasized the effects of persistent susceptibility and connectivity heterogeneity on epidemic dynamics and herd-immunity thresholds~\citep{GomesEtAl2022}. More recently, Diekmann and Inaba introduced a systematic procedure for incorporating separable static heterogeneity into compartmental epidemic models, thereby unifying a broad class of trait-structured formulations~\citep{DiekmannInaba2023}. Related studies have examined the effects of susceptibility variance, frailty, contact classes, and dynamically varying social activity~\citep{MargheriEtAl2015,Szapudi2020,TkachenkoEtAl2021}.

The second approach relies on an individual-level stochastic description of
the epidemic process. In the classical  Sellke construction, each initially susceptible
individual is assigned a random resistance threshold and becomes infected when
the cumulative population infection pressure exceeds that threshold
\citep{Sellke1983}. By separating individual resistance from the common force
of infection, the construction provides a particularly convenient framework
for large-population asymptotics. Persistent individual heterogeneity can be incorporated into the Sellke construction by allowing the resistance distribution to depend on an individual-specific risk or activity characteristic; related constructions have been developed for discrete heterogeneous contact rates, for instance,  in \cite{House2014}. Here we consider a susceptibility-specific formulation in which each individual is assigned a fixed, nonnegative susceptibility characteristic that determines the distribution of the resistance threshold. Unlike contact-activity models, in which heterogeneity may influence both acquisition and transmission, the present framework restricts heterogeneity to susceptibility. This yields a
natural stochastic model of heterogeneous transmission risk. As already indicated, among the
best-studied explicit deterministic reductions is the one based on the Gamma family, whose
Laplace transform gives rise to a power-law incidence function
\citep{Novozhilov2008,GomesEtAl2022,DiekmannInaba2023}.

Against this background, the purpose of this paper is to establish the connection between the heterogeneous Sellke construction and deterministic heterogeneous SIR models for a broad class of 
susceptibility distributions. The latter are typically formulated directly at the population level, whereas the stochastic construction considered here begins with a finite population 
carrying fixed individual susceptibility characteristics or marks. We show that, under mild moment conditions, the corresponding epidemic process converges in a law-of-large-numbers sense to deterministic 
heterogeneous SIR dynamics whose structure is governed by the Laplace transform of the susceptibility distribution. In this way, 
the generalized Sellke construction provides an individual-level probabilistic foundation for a broad class of heterogeneous SIR models.

Let \(F\) denote the law of a nonnegative
susceptibility mark \(\Lambda\), and let 
\[
G(a)=\int_{[0,\infty)}e^{-a\lambda}\,F(d\lambda)
     =\E[e^{-a\Lambda}],
\] denote  its  Laplace transform. Then  the limiting susceptible proportion satisfies
\[
s(t)=s_0G(A(t)),
\qquad
A(t)=\tau\int_0^t\iota(u)\,du,
\]
where \(A(t)\) is cumulative infection pressure and  $\iota(t)$ is  the limiting  infected proportion. The corresponding incidence
coefficient is
\[
\overline{\lambda}(A(t))
=
-\frac{G'(A(t))}{G(A(t))}
=
\frac{\E[\Lambda e^{-A(t)\Lambda}]}
     {\E[e^{-A(t)\Lambda}]},
\]
the mean susceptibility under the exponentially tilted distribution of marks
among individuals who remain susceptible. Thus, the deterministic limit does
not replace \(\Lambda\) by its initial mean. It retains the evolving
composition of the susceptible population.

This representation clarifies both the probabilistic mechanism and the
question of analytic tractability. The epidemic always admits a 
description solely in terms of cumulative pressure. A closed system in the usual
compartment variables is available whenever the relevant transform relations
can be inverted or represented by finitely many auxiliary variables. The existence of an exact finite-dimensional reduction is therefore governed by the analytic structure of the Laplace transform, rather than by any particular susceptibility distribution, including the Gamma family.

The main contributions of this paper are threefold. First, for an arbitrary
nonnegative susceptibility distribution with finite mean, we prove
uniform-on-compact convergence of the finite  epidemic with heterogeneous susceptibility to a
deterministic limit. Second, we characterize this limit by deriving the
incidence function, initial-growth criterion, reproduction number, final-size
equation, and the selection identity (Proposition~\ref{prop:selection}), which
describes how the average susceptibility of the remaining susceptible
population changes over the course of the epidemic. Third, we show that the
limiting equations naturally recover and extend classical deterministic models
of heterogeneous susceptibility. They reproduce the familiar power-law
incidence associated with Gamma-distributed susceptibilities, provide a simple
probabilistic interpretation of dynamic selection for general susceptibility
distributions, and reveal additional susceptibility families that lead to
finite-dimensional ODE systems.

The examples that follow are intended to illustrate this general framework
rather than simply catalogue alternative incidence functions. Each
susceptibility distribution corresponds to a familiar modeling paradigm.
Gamma distributions recover the classical continuous heterogeneous-SIR model,
while finite discrete distributions provide an exact stochastic foundation for
multigroup and risk-class models. Poisson and negative-binomial distributions
naturally describe latent counts of exposure opportunities or risk factors,
including the possibility of a completely resistant subpopulation, whereas the
inverse-Gaussian distribution provides a continuous non-Gamma example with an
explicit finite-dimensional reduction. Together, these examples show that
tractable heterogeneous epidemic models can be generated systematically from
susceptibility distributions whose Laplace transforms possess suitable
differential or algebraic structure.

Throughout the paper, heterogeneity is intrinsic rather than environmental:
each individual is assigned a susceptibility mark at the beginning of the
epidemic, and this mark remains fixed until infection. The susceptibility
distribution therefore evolves only through the preferential infection of
high-risk individuals. This dynamic-selection mechanism is fundamentally
different from models in which transmission rates fluctuate randomly over
time, representing environmental variability, even when the same marginal
distribution is used. %\cite{VaupelYashin1985,Gomes2025}.
We also restrict
heterogeneity to susceptibility and assume exponentially distributed
infectious periods. Allowing more general infectious-period distributions
naturally leads to infection-age, renewal, Volterra, or phase-type models
\citep{FengThieme2000,ForienPangPardoux2021}, while correlated susceptibility
and infectiousness would require joint marked populations. These extensions
lie beyond the scope of the present paper.

The remainder of the paper is organized as follows.
Section~\ref{sec:model} defines the finite heterogeneous-susceptibility epidemic, which we refer to as the mixed-Sellke model, and introduces the Laplace-transform representation of the resistance distribution.
Section~\ref{sec:limit} states the fluid-limit theorem and develops its
deterministic consequences, including selection, reproduction numbers, and
final size. Section~\ref{sec:examples} presents the five representative
susceptibility families and discusses their relation to existing literature
and applications. 
 Section~\ref{sec:discussion} summarizes the scope and limitations of
the framework.  The proof of the
general fluid-limit theorem is provided in the Appendix.

\section{Finite-population model}\label{sec:model}

We first specify the finite stochastic epidemic on a common probability
space. The construction makes explicit the two independent mechanisms in the
model: persistent heterogeneity in susceptibility and homogeneous exponential
recovery. This formulation also yields exact identities that will later reduce
the proof to the analysis of a one-dimensional  integral equation. 

For population size \(N\), let { $S_N(t), I_N(t),$ and $R_N(t)$ denote the numbers of susceptible, infected and removed individuals at time $t$, respectively, so that }
\[
S_N(t)+I_N(t)+R_N(t)=N
\]
and define the corresponding proportions
\[
s_N(t)=\frac{S_N(t)}{N},\qquad
\iota_N(t)=\frac{I_N(t)}{N},\qquad
r_N(t)=\frac{R_N(t)}{N}.
\]

Assume
\[
\bigl(s_N(0),\iota_N(0),r_N(0)\bigr)
\xlongrightarrow{\Pp}
(s_0,\iota_0,r_0),
\]
where
\[
s_0+\iota_0+r_0=1,\qquad s_0>0,\qquad \iota_0>0.
\]

For each of the $S_N(0)$ initially susceptible individuals, assign a nonnegative
susceptibility mark $\Lambda_j$, $j=1,\ldots,S_N(0)$. The marks are independent
and identically distributed with common law $\mu$ on $[0,\infty)$ and are
independent of the initial compartment counts:
\[
\Lambda_j \stackrel{\mathrm{iid}}{\sim} \mu,
\qquad j=1,\ldots,S_N(0).
\]
Let $\Lambda$ denote a generic random variable with law $\mu$, and assume
\[
m_1:=\E[\Lambda]<\infty.
\]

Conditionally on \(\Lambda_j\), assign an independent Sellke resistance
threshold
\[
Q_j\mid \Lambda_j\sim \operatorname{Exp}(\Lambda_j),
\]
so that
\[
\Pp(Q_j>a\mid \Lambda_j=\lambda)=e^{-\lambda a},
\qquad a\geq 0.
\]

Define the cumulative infection pressure by
\[
A_N(t)=\tau\int_0^t\iota_N(u)\,du,
\qquad \tau>0.
\]
An initially susceptible individual remains susceptible at time \(t\) exactly
when
\[
Q_j>A_N(t).
\]

Each infected individual has an independent infectious period distributed as
\(\operatorname{Exp}(\gamma)\), where \(\gamma>0\). Equivalently, the recovery
process has the random time-change representation (see, e.g., \cite{AndersonKurtz2015})
\[
R_N(t)
=
R_N(0)+
Y_R\left(
N\gamma\int_0^t\iota_N(u)\,du
\right),
\]
where \(Y_R\) is a unit-rate Poisson process independent of the susceptibility
marks, resistance thresholds  and  initial compartment counts. Note that the unit-rate Poisson clock
is evaluated at its integrated stochastic intensity
\(N\gamma\int_0^t\iota_N(u)\,du\)
\citep{Kurtz1980,AndersonKurtz2015}.

\subsection{Resistance distribution and Laplace transform}

The unconditional resistance law is obtained by mixing exponential
thresholds over the susceptibility distribution. Its survival function is
therefore the Laplace transform of that distribution. This observation is the
main analytic bridge between the individual-level marks and the deterministic
population limit.

Define the Laplace transform
\begin{equation}\label{eq:Ga}
G(a)
=
\E[e^{-a\Lambda}]
=
\int_{[0,\infty)} e^{-a\lambda}\,\mu(d\lambda).
\end{equation}
Since \(\E[\Lambda]<\infty\), differentiation under the expectation is justified
by the dominated convergence theorem, yielding
\[
G'(a)
=
-\E[\Lambda e^{-a\Lambda}].
\]
Therefore,
\[
|G'(a)|
\leq
\E[\Lambda]
=
m_1.
\]
Hence, \(G\) is globally Lipschitz on \([0,\infty)\).

\section{Deterministic limit and structural reduction}\label{sec:limit}

The theorem below states that, on every bounded time interval, the finite epidemic is asymptotically described by a deterministic trajectory. In the terminology of stochastic population processes and chemical reaction networks, such a law-of-large-numbers limit for the population-scaled process is commonly called a \emph{fluid limit}: as the population size tends to infinity, the  population-scaled stochastic processes converge uniformly on compact time intervals to a deterministic trajectory, as described, for instance, in  \cite{EthierKurtz1986, AndersonKurtz2015}. A useful feature of the Sellke representation is that the entire fluid limit can be characterized through a single
equation for the cumulative infection pressure $A$; the compartment proportions are then recovered algebraically. Throughout, $\left\|\cdot\right\|$  denotes any norm on the relevant finite-dimensional Euclidean space.

\begin{theorem}[General mixed-Sellke fluid limit]\label{thm:1}
Under the stochastic construction and notation of Section~\ref{sec:model}, for every finite \(T>0\),
\[
\sup_{0\leq t\leq T}
\left\|
\bigl(s_N(t),\iota_N(t),r_N(t),A_N(t)\bigr)
-
\bigl(s(t),\iota(t),r(t),A(t)\bigr)
\right\|
\xlongrightarrow{\Pp}0
\]
as $N\to\infty.$

The deterministic limit is characterized by the  initial-value problem
\[
\dot A(t)
=
\tau\bigl[1-r_0-s_0G(A(t))\bigr]
-\gamma A(t),
\qquad
A(0)=0,
\]
together with
\[
s(t)=s_0G(A(t)),
\]
\[
r(t)=r_0+\frac{\gamma}{\tau}A(t),
\]
and
\[
\iota(t)=1-s(t)-r(t).
\]

Equivalently,
\[
A(t)=\tau\int_0^t\iota(u)\,du.
\]
\end{theorem}

The proof of the general fluid-limit theorem is deferred to Appendix~\ref{app:general-proof}.

\subsection{Differential equation representation}

The one-dimensional  formulation is convenient for the proof, but the epidemiological
meaning of the limit is clearest in compartmental differential-equation form. The incidence
coefficient is not the initial mean susceptibility. It is the mean
susceptibility under the distribution of marks among individuals who have
survived the accumulated infection pressure.

Recall from Section~\ref{sec:intro}  the exponentially tilted mean susceptibility
\[
\overline{\lambda}(a)
=
\frac{\E[\Lambda e^{-a\Lambda}]}
{\E[e^{-a\Lambda}]}
=
-\frac{G'(a)}{G(a)}.
\]
Since \(G(a)>0\) for every finite \(a\geq 0\), this quantity is well defined.

The limiting epidemic satisfies
\[
\dot s(t)
=
-\tau\overline{\lambda}(A(t))s(t)\iota(t),
\]
\[
\dot\iota(t)
=
\tau\overline{\lambda}(A(t))s(t)\iota(t)
-\gamma\iota(t),
\]
\[
\dot r(t)=\gamma\iota(t),
\]
and
\[
\dot A(t)=\tau\iota(t),
\]
with
\[
(s(0),\iota(0),r(0),A(0))
=
(s_0,\iota_0,r_0,0).
\]

Indeed,
\[
\begin{aligned}
\dot s(t)
&=
s_0G'(A(t))\dot A(t)\\
&=
-s_0\E[\Lambda e^{-A(t)\Lambda}]
\,\tau\iota(t)\\
&=
-\tau
\frac{\E[\Lambda e^{-A(t)\Lambda}]}
{\E[e^{-A(t)\Lambda}]}
s(t)\iota(t).
\end{aligned}
\]

The tilted susceptibility distribution among individuals still susceptible
after cumulative pressure \(a\) is
\begin{equation}\label{eq:mu_a}
\mu_a(d\lambda)
=
\frac{e^{-a\lambda}\mu(d\lambda)}{G(a)}.
\end{equation}
Accordingly,
\begin{equation}\label{eq:la_a}
\overline{\lambda}(a)=\E_a[\Lambda], 
\end{equation}
where $\E_a$ is expectation with respect to the measure $\mu_a$. 

\subsection{Monotonicity of effective susceptibility}\label{sec:effective}

Persistent susceptibility heterogeneity induces selection within the
susceptible compartment. The following identity is the basic quantitative
statement used in the later interpretation.

\begin{proposition}[Selection identity]\label{prop:selection}
Assume \(\E[\Lambda^2]<\infty\). For every \(a\geq0\) with \(G(a)>0\),
\[
\overline{\lambda}'(a)=-\operatorname{Var}_a(\Lambda)\leq0.
\]
Hence, the effective susceptibility is nonincreasing in cumulative infection
pressure, and is strictly decreasing whenever the tilted law \(\mu_a\) is
nondegenerate.
\end{proposition}

\begin{proof}
Since \(\overline{\lambda}(a)=-G'(a)/G(a)\),
\[
\overline{\lambda}'(a)
=-\frac{G''(a)G(a)-G'(a)^2}{G(a)^2}
=-\left(\E_a[\Lambda^2]-\E_a[\Lambda]^2\right).
\]
\end{proof}

\subsection{Closed system in the compartment variables}

Although the cumulative pressure \(A\) has a direct Sellke interpretation, it
can be eliminated whenever the Laplace transform is invertible on the relevant
range. This gives a closed three-dimensional system in the standard
compartment variables, with a generally nonlinear incidence coefficient.

Let
\( p_0 =\Pr(\Lambda=0).\)
If \(p_0<1\), then \(G\) is strictly decreasing from \(1\) to \(p_0\) and hence admits an inverse on \((p_0,1]\).
Since according to Theorem~\ref{thm:1}
\[
\frac{s(t)}{s_0}=G(A(t)),
\]
we may write
\[
A(t)=G^{-1}\left(\frac{s(t)}{s_0}\right).
\]

Define
\[
\psi(x)
=
-\frac{
G'\!\left(G^{-1}(x)\right)
}{x},
\qquad p_0<x\leq 1.
\]
The limiting dynamics can then be represented as the closed system
\[
\dot s
=
-\tau
\psi\left(\frac{s}{s_0}\right)s\iota,
\]
\[
\dot\iota
=
\tau
\psi\left(\frac{s}{s_0}\right)s\iota
-\gamma\iota,
\]
\[
\dot r=\gamma\iota.
\]

Equivalently,
\[
\psi\left(\frac{s}{s_0}\right)
=
\overline{\lambda}\left(
G^{-1}\left(\frac{s}{s_0}\right)
\right).
\]
For a general susceptibility distribution, this function need not have an
elementary closed form. Section~\ref{sec:examples} presents several
susceptibility families for which explicit formulas and finite-dimensional
reductions are available.

\subsection{Initial growth,  reproduction number and final size }

At the start of the epidemic no selection has yet occurred, so the effective
susceptibility equals the ordinary population mean. The initial growth
criterion therefore depends only on the first moment of the susceptibility
distribution, even though later dynamics depend on its full Laplace transform.
At \(t=0\),
$A(0)=0,G(0)=1,$
and therefore
$
\overline{\lambda}(0)=\E[\Lambda].
$
It follows that
\[
\dot\iota(0)
=
\left(
\tau s_0\E[\Lambda]-\gamma
\right)\iota_0.
\]

The initial effective reproduction number is
\[
\mathcal R_{\mathrm{eff}}(0)
=
\frac{\tau s_0\E[\Lambda]}{\gamma}.
\]
When \(s_0\) is close to one, the corresponding basic reproduction number is
\[
\mathcal R_0
=
\frac{\tau\E[\Lambda]}{\gamma}.
\]

%\subsection{Final-size relation}

The terminal epidemic size can also be expressed through the Laplace transform.
This relation generalizes the standard homogeneous final-size equation and
shows precisely where the full susceptibility distribution enters the
eventual depletion of susceptibles.

Let
\[
A_\infty=\lim_{t\to\infty}A(t).
\]
Since \(\iota(t)\to 0\) as \(t\to\infty \), the limiting cumulative infection pressure satisfies
\[
1-r_0-s_0G(A_\infty)
-\frac{\gamma}{\tau}A_\infty
=0.
\]
Equivalently,
\[
A_\infty
=
\frac{\tau}{\gamma}
\left[
1-r_0-s_0G(A_\infty)
\right].
\]

The limiting susceptible and removed proportions are
\[
s_\infty=s_0G(A_\infty)
\]
and
\[
r_\infty=r_0+\frac{\gamma}{\tau}A_\infty.
\]
Thus, the final epidemic size depends on the susceptibility distribution only
through its Laplace transform \(G\).
 This representation also yields a
simple comparison principle: among susceptibility distributions with a fixed
mean, the homogeneous population produces the largest epidemic final size.

\begin{proposition}[Final-size reduction through susceptibility heterogeneity]
\label{thm:final_size}
Assume that
\[
\E[\Lambda]=1,
\]
and compare the heterogeneous epidemic with the homogeneous model
\(\Lambda\equiv 1\), keeping \(\tau\), \(\gamma\), and the initial condition
\((s_0,\iota_0,r_0)\) fixed. Let \(A_\infty\) and \(r_\infty\) denote the
limiting cumulative infection pressure and removed proportion in the
heterogeneous model, and let \(A_\infty^{\rm hom}\) and
\(r_\infty^{\rm hom}\) denote the corresponding quantities in the
homogeneous model. Then
\[
A_\infty\leq A_\infty^{\rm hom},
\qquad
r_\infty\leq r_\infty^{\rm hom}.
\]
If the law of \(\Lambda\) is nondegenerate and \(\iota_0>0\), the inequalities
are strict.
\end{proposition}

\begin{proof}
For every \(a\geq 0\), the function
\[
\lambda\longmapsto e^{-a\lambda}
\]
is convex. Hence, Jensen's inequality gives
\[
G(a)
=
\E\!\left[e^{-a\Lambda}\right]
\geq
e^{-a\E[\Lambda]}
=
e^{-a}.
\]
Therefore, for every \(a\geq0\),
\[
\tau\left[1-r_0-s_0G(a)\right]-\gamma a
\leq
\tau\left[1-r_0-s_0e^{-a}\right]-\gamma a.
\]
The left-hand side is the vector field describing the derivative of the cumulative infection pressure \(A(t)\) in the heterogeneous model introduced in Theorem~\ref{thm:1}, whereas the right-hand side is the corresponding vector field for the homogeneous model.
 Since both processes
start from
\[
A(0)=A^{\rm hom}(0)=0,
\]
the  ODE comparison principle implies
\[
A(t)\leq A^{\rm hom}(t),
\qquad t\geq0.
\]
Passing to the limit \(t\to\infty\) yields
\[
A_\infty\leq A_\infty^{\rm hom}.
\]
Since
\[
r_\infty
=
r_0+\frac{\gamma}{\tau}A_\infty,
\]
we also obtain
\[
r_\infty\leq r_\infty^{\rm hom}.
\]

If \(\Lambda\) is nondegenerate, then strict convexity gives
\[
G(a)>e^{-a},
\qquad a>0.
\]
When \(\iota_0>0\), positive cumulative infection pressure is generated, so
the comparison is strict and hence
\[
A_\infty<A_\infty^{\rm hom},
\qquad
r_\infty<r_\infty^{\rm hom}.
\]
\end{proof}

\begin{remark}
In relation to Proposition~\ref{thm:final_size}, a word of caution is warranted. Although its conclusion is consistent with \citet[Section~3]{Andersson1998heterogeneity} and has been recognized in the applied probability literature since at least the mid-1980s \citep{Ball1985detStoch}, its validity depends importantly on how the heterogeneous population is compared with its homogeneous counterpart.
In particular, the conclusion may depend on the choice of quantity held fixed across the two populations.

\vspace{.2in}

To illustrate the issue raised in the remark, consider 
for example, instead of matching the mean susceptibility, the matching of the mean Sellke threshold. In the heterogeneous model the threshold 
\(
Q \mid \Lambda \sim \operatorname{Exp}(\Lambda),
\)
so that
\[
\E[Q]=\E[1/\Lambda].
\]
with the equality interpreted in the extended sense when $\Pr(\Lambda=0)>0$. If the homogeneous population is constructed through the Sellke representation with i.i.d. exponential thresholds having constant rate $\lambda^\ast$, then matching the mean threshold gives
\begin{align*}
    \frac{1}{\lambda^\ast}
    = \E[Q]
    = \E[1/\Lambda],
    \qquad\text{hence}\qquad
    \lambda^\ast
    = \frac{1}{\E[1/\Lambda]}.
\end{align*}
Thus $\lambda^\ast$ is the harmonic mean of $\Lambda$, and the corresponding homogeneous incidence term would be
\(
\tau \lambda^\ast s(t)\iota(t)
\)
rather than
\(
\tau \E[\Lambda] s(t)\iota(t).
\)
With this normalization, the homogeneous model may have a substantially smaller final epidemic size than the corresponding heterogeneous model, simply because
\(
1/{\E[1/\Lambda]}
\)
can be much smaller than $\E[\Lambda]$. For distributions having an atom at zero, such as the Poisson example,
\(
\E\!\left[1/{\Lambda}\right]=\infty,
\)
and hence we may take \(\lambda^*=0\). The same occurs when 
\(\Lambda\sim\mathrm{Gamma}(\alpha,\kappa)\) when the shape parameter satisfies
\(\alpha\leq 1\) (but not when  \(\alpha>1\)).

This sensitivity to the comparison criterion is also reflected in the multitype epidemic model with varying susceptibility, corresponding to the finite discrete setting considered here. As shown in \citet[Observation~1]{Andersson1998heterogeneity}, when the disease is sufficiently contagious, the homogeneous model produces the largest total epidemic size, whereas for less contagious diseases a heterogeneous population may produce the larger epidemic. Moreover, if one compares instead the probability of a large outbreak, the homogeneous model is always minimal, irrespective of the degree of contagiousness \cite[p.~656]{Andersson1998heterogeneity}.

Thus Proposition~\ref{thm:final_size} should be interpreted relative to the particular normalization used there, namely comparison at fixed mean susceptibility. More generally, statements that heterogeneity increases or decreases epidemic size are meaningful only after specifying both the quantity being compared and the way in which the heterogeneous and homogeneous populations are matched.
\end{remark}

\subsection{Dynamic selection and effective susceptibility}

The limiting equations have a direct interpretation in the theory of
heterogeneous populations. They describe a law of large numbers with
\emph{dynamic selection}, rather than a static averaging of individual
susceptibilities.

At cumulative infection pressure \(a\), an individual with susceptibility
\(\lambda\) remains susceptible with probability \(e^{-a\lambda}\).
Consequently, the susceptibility distribution among the remaining
susceptibles is
\(
\mu_a(d\lambda) \) given by \eqref{eq:mu_a}. 
The incidence coefficient in the limiting equations is therefore
\(
\overline{\lambda}(a)
\) given by \eqref{eq:la_a}, i.e., 
the mean susceptibility of the population that has not yet been infected.
Individuals with larger susceptibility marks are removed from the susceptible
pool earlier, even though every individual's mark remains fixed.

This mechanism is analogous to the ``ruses'' of heterogeneity 
 in survival populations described by \cite{VaupelYashin1985}. In their
setting, frailer or higher-risk individuals die or exit earlier, so the
composition of the surviving cohort becomes increasingly robust. The
population hazard may consequently decline, flatten, cross another cohort's
hazard, or follow a nonmonotone trajectory even when subgroup hazards are
simple.   According to Proposition~\ref{prop:selection}, in the epidemic setting the corresponding effect is a decline in the
effective susceptibility of the remaining susceptible population. 
The apparent decline in transmission intensity is thus not caused by a change
in any individual's susceptibility. It results from the selective depletion
of individuals with large \(\Lambda\).

The distinction between persistent individual heterogeneity and environmental
stochasticity is also relevant here.
In the present
model, \(\Lambda_j\) is an individual-specific mark assigned once and retained
throughout the susceptible lifetime. A model in which a transmission parameter or
growth rate is repeatedly redrawn over time, either independently across individuals or synchronously as in the case of environmental stochasticity would behave differently (e.g., contrast \cite{mackenzie2025impact} with \cite{Gillespie1973}). Even if the same marginal distributions were used
in the two constructions, the resulting population dynamics would generally
not agree. Persistent heterogeneity creates selection and an evolving
composition, whereas environmental fluctuations act simultaneously on the
whole population and are averaged over time.

This distinction also explains why replacing \(\Lambda\) by
\(\E[\Lambda]\) is generally incorrect beyond the initial phase. Such a
replacement would produce the homogeneous incidence term
\[
\tau\E[\Lambda]s(t)\iota(t),
\]
which assumes that the susceptibility distribution among susceptibles remains
unchanged. The correct limiting incidence is
\[
\tau\overline{\lambda}(A(t))s(t)\iota(t),
\]
where \(\overline{\lambda}(A(t))\leq \E[\Lambda]\) and typically decreases
strictly during the epidemic. The Laplace-transform representation therefore
retains precisely the compositional information that would be lost under
naive mean-field averaging.

There is also a useful conditional interpretation at the individual level.
Given \(\Lambda_j=\lambda\), the Sellke threshold remains
\[
Q_j\mid\{\Lambda_j=\lambda\}\sim\operatorname{Exp}(\lambda).
\]
If
\[
T_j=\inf\{t\geq0:A(t)\geq Q_j\}
\]
denotes the limiting infection time, then
\[
\Pp(T_j>t\mid\Lambda_j=\lambda)
=
e^{-\lambda A(t)}
=
\exp\left\{
-\tau\lambda\int_0^t\iota(u)\,du
\right\}.
\]
Thus, the resistance is exponential on the cumulative-pressure scale, while
the induced infection time is generally nonexponential on calendar time. The
conditional infection hazard is
\[
h_\lambda(t)=\lambda\dot A(t)=\tau\lambda\iota(t).
\]
The deterministic limit in Theorem~\ref{thm:1}  may therefore be viewed as the population-level
aggregation of these conditional survival laws, with the aggregation weights
changing endogenously through selection.

\section{Some susceptibility distributions admitting finite ODE closures}\label{sec:examples}

The general mixed-Sellke limit is determined by the Laplace transform
\[
G(a)=\E[e^{-a\Lambda}]
\]
and by the corresponding tilted mean susceptibility
\[
\overline{\lambda}(a)=-\frac{G'(a)}{G(a)}.
\]
Whenever the relation \(x=G(a)\), with \(x=s/s_0\), can be inverted explicitly,
the cumulative pressure can be eliminated and the model becomes a closed system
in the usual compartment variables.  The following five examples illustrate
several qualitatively different forms of heterogeneity: continuous and
discrete, unimodal and multimodal, with or without a completely resistant
subpopulation (for a thorough comparison between gamma and lognormal forms of heterogeneity, see \cite{mohammed2026exploration}).

Throughout this section, the reduced equations have the form
\[
\dot s=-\tau\psi\!\left(\frac{s}{s_0}\right)s\iota,
\qquad
\dot\iota=\tau\psi\!\left(\frac{s}{s_0}\right)s\iota-\gamma\iota,
\qquad
\dot r=\gamma\iota,
\]
where
\[
\psi(x)=\overline{\lambda}\bigl(G^{-1}(x)\bigr).
\]

\subsection{Gamma: the benchmark power-law model}

Let
\[
\Lambda\sim\operatorname{Gamma}(\alpha,\kappa),
\]
where \(\alpha>0\) is the shape and \(\kappa>0\) is the rate.  Then
\[
G(a)=\left(\frac{\kappa}{\kappa+a}\right)^\alpha,
\qquad
\overline{\lambda}(a)=\frac{\alpha}{\kappa+a}.
\]
Since
\[
\frac{s}{s_0}=\left(\frac{\kappa}{\kappa+A}\right)^\alpha,
\]
we obtain
\[
\psi\!\left(\frac{s}{s_0}\right)
=\frac{\alpha}{\kappa}
\left(\frac{s}{s_0}\right)^{1/\alpha}.
\]
Thus
\[
\dot s
=-\tau\frac{\alpha}{\kappa}
\left(\frac{s}{s_0}\right)^{1/\alpha}s\iota,
\]
\[
\dot\iota
=\tau\frac{\alpha}{\kappa}
\left(\frac{s}{s_0}\right)^{1/\alpha}s\iota
-\gamma\iota.
\]
Equivalently, the incidence is proportional to
\[
s^{1+1/\alpha}\iota.
\]
We consider this model as a standard benchmark because it reproduces the classical
power-law incidence arising in deterministic SIR models with heterogeneous
susceptibility \citep{Novozhilov2008,GomesEtAl2022}. The homogeneous model is
recovered in the limit \(\alpha \to \infty\) while the mean susceptibility
\(\alpha/\kappa\) is held fixed.

\subsection{Poisson: logarithmic incidence and a resistant subpopulation}

Let
\[
\Lambda\sim\operatorname{Poisson}(\theta),
\qquad \theta>0.
\]
The value \(\Lambda=0\) represents complete resistance; equivalently, the
corresponding Sellke threshold is infinite.  The Laplace transform and tilted
mean are
\[
G(a)=\exp\{\theta(e^{-a}-1)\},
\qquad
\overline{\lambda}(a)=\theta e^{-a}.
\]
Writing \(x=s/s_0\),
\[
\log x=\theta(e^{-A}-1),
\]
so that
\[
\psi(x)=\theta+\log x.
\]
The reduced system is therefore
\[
\dot s
=-\tau\left[\theta+\log\left(\frac{s}{s_0}\right)\right]s\iota,
\]
\[
\dot\iota
=\tau\left[\theta+\log\left(\frac{s}{s_0}\right)\right]s\iota
-\gamma\iota.
\]
The dynamically accessible range is
\[
e^{-\theta}\leq \frac{s(t)}{s_0}\leq1,
\]
because
\[
\lim_{a\to\infty}G(a)=\Pp(\Lambda=0)=e^{-\theta}.
\]
Hence at least the fraction \(s_0e^{-\theta}\) of the initially susceptible
population remains uninfected even under arbitrarily large cumulative
infection pressure.  This example gives a simple logarithmic incidence law
and shows that a discrete susceptibility distribution still produces smooth
macroscopic dynamics.

\subsection{Negative-binomial: overdispersion and a difference-of-powers incidence}

Let
\[
\Lambda\sim\operatorname{NegBin}(r,p),
\qquad r>0,\quad 0<p<1,
\]
with probability mass function
\[
\Pp(\Lambda=k)
=\frac{\Gamma(k+r)}{\Gamma(r)k!}p^r(1-p)^k,
\qquad k=0,1,2,\ldots.
\]
Thus,
\[
\E[\Lambda]=\frac{r(1-p)}{p},
\qquad
\operatorname{Var}(\Lambda)=\frac{r(1-p)}{p^2},
\]
so the variance exceeds the mean.  Its Laplace transform is
\[
G(a)
=\left(
\frac{p}{1-(1-p)e^{-a}}
\right)^r,
\]
and
\[
\overline{\lambda}(a)
=\frac{r(1-p)e^{-a}}{1-(1-p)e^{-a}}.
\]
If \(x=s/s_0\), then
\[
x^{-1/r}=\frac{1-(1-p)e^{-A}}{p},
\]
which gives
\[
\psi(x)
=r\left(\frac{x^{1/r}}{p}-1\right).
\]
Consequently,
\[
\dot s
=-\tau r\left[
\frac{1}{p}\left(\frac{s}{s_0}\right)^{1/r}-1
\right]s\iota,
\]
\[
\dot\iota
=\tau r\left[
\frac{1}{p}\left(\frac{s}{s_0}\right)^{1/r}-1
\right]s\iota-\gamma\iota.
\]
The infection term may equivalently be written as the difference of two
powers,
\[
\tau r\left[
\frac{s^{1+1/r}}{p\,s_0^{1/r}}-s
\right]\iota.
\]
This gives a discrete, overdispersed alternative to Poisson heterogeneity.  It
also contains a resistant atom,
\[
\Pp(\Lambda=0)=p^r,
\]
so the accessible susceptible range satisfies
\[
p^r\leq \frac{s(t)}{s_0}\leq1.
\]

\subsection{Finite discrete: arbitrary multimodality}

Let \(\Lambda\) have finite support
\[
\lambda_1,\ldots,\lambda_m\geq 0
\]
with probabilities
\[
p_1,\ldots,p_m>0,
\qquad
\sum_{j=1}^m p_j=1.
\]
Then
\[
G(a)=\sum_{j=1}^m p_j e^{-\lambda_j a}
\]
and
\[
\overline{\lambda}(a)
=
\frac{\sum_{j=1}^m \lambda_j p_j e^{-\lambda_j a}}
{\sum_{j=1}^m p_j e^{-\lambda_j a}}.
\]

Although \(G^{-1}\) generally has no elementary expression, the limiting model
admits an exact finite-dimensional closure by retaining the susceptible fraction
in each susceptibility class. Define
\[
s_j(t)=s_0 p_j e^{-\lambda_j A(t)},
\qquad j=1,\ldots,m.
\]
Then
\[
s(t)=\sum_{j=1}^m s_j(t),
\]
and, since \(\dot A(t)=\tau\iota(t)\),
\[
\dot s_j
=
-\tau\lambda_j s_j\iota,
\qquad j=1,\ldots,m.
\]
Consequently, the limiting epidemic is governed by
\[
\dot s_j
=
-\tau\lambda_j s_j\iota,
\qquad j=1,\ldots,m,
\]
\[
\dot\iota
=
\tau\iota\sum_{j=1}^m \lambda_j s_j
-\gamma\iota,
\qquad
\dot r=\gamma\iota.
\]
Thus the general limit of Theorem~\ref{thm:1} reduces in the finite-support
case to an exact \((m+2)\)-dimensional ODE system. The formulation accommodates
arbitrary multimodality, separated risk classes, and any prescribed collection
of susceptibility atoms. If one or more \(\lambda_j\) are equal to zero, their
combined probability mass corresponds to a completely resistant subpopulation.

This finite discrete case is related to the heterogeneous-contact formulation
of \citet{House2014}, who used a generalized Sellke construction for populations
with discrete contact rates. The interpretation and scope, however, are
different. In \citet{House2014}, the individual mark represents contact activity
and enters the transmission mechanism through heterogeneous mixing, whereas in
the present model the marks represent susceptibility alone and infectiousness
remains homogeneous. Moreover, the finite discrete system above is not the
starting point of our asymptotic argument, but rather a specialization of
Theorem~\ref{thm:1}, which applies to an arbitrary nonnegative susceptibility
distribution with finite mean. Indeed, the multigroup closure follows directly
from the general Laplace-transform representation above,  
with the individual class equations obtained by retaining separately the
finitely many components of the susceptible population.

\subsection{Inverse-Gaussian logarithmic-incidence model}
Let
\[
\Lambda\sim\operatorname{IG}(m,\nu),
\qquad m>0,\quad \nu>0,
\]
where \(m\) is the mean and \(\nu\) is the shape parameter.  Its Laplace
transform is
\[
G(a)
=\exp\left\{
\frac{\nu}{m}
\left(1-\sqrt{1+\frac{2m^2a}{\nu}}\right)
\right\}.
\]
Differentiation gives
\[
\overline{\lambda}(a)
=\frac{m}{\sqrt{1+2m^2a/\nu}}.
\]
Writing again \(x=s/s_0\),
\[
\log x
=\frac{\nu}{m}
\left(1-\sqrt{1+\frac{2m^2A}{\nu}}\right),
\]
and hence
\[
\sqrt{1+\frac{2m^2A}{\nu}}
=1-\frac{m}{\nu}\log x.
\]
Therefore,
\[
\psi(x)
=\frac{m}{1-(m/\nu)\log x}.
\]
The reduced system becomes
\[
\dot s
=-\tau
\frac{m}{1-(m/\nu)\log(s/s_0)}s\iota,
\]
\[
\dot\iota
=\tau
\frac{m}{1-(m/\nu)\log(s/s_0)}s\iota
-\gamma\iota.
\]

This yields a continuous non-Gamma example with a rational-logarithmic
incidence coefficient. We refer to the resulting reduced system as the
\emph{inverse-Gaussian logarithmic-incidence model}. Unlike the power-law
incidence obtained from Gamma susceptibility, the effective susceptibility
depends rationally on \(\log(s/s_0)\).

\vspace{.22in}

\begin{figure}[htbp]
    \centering
    \includegraphics[width=\textwidth]
    {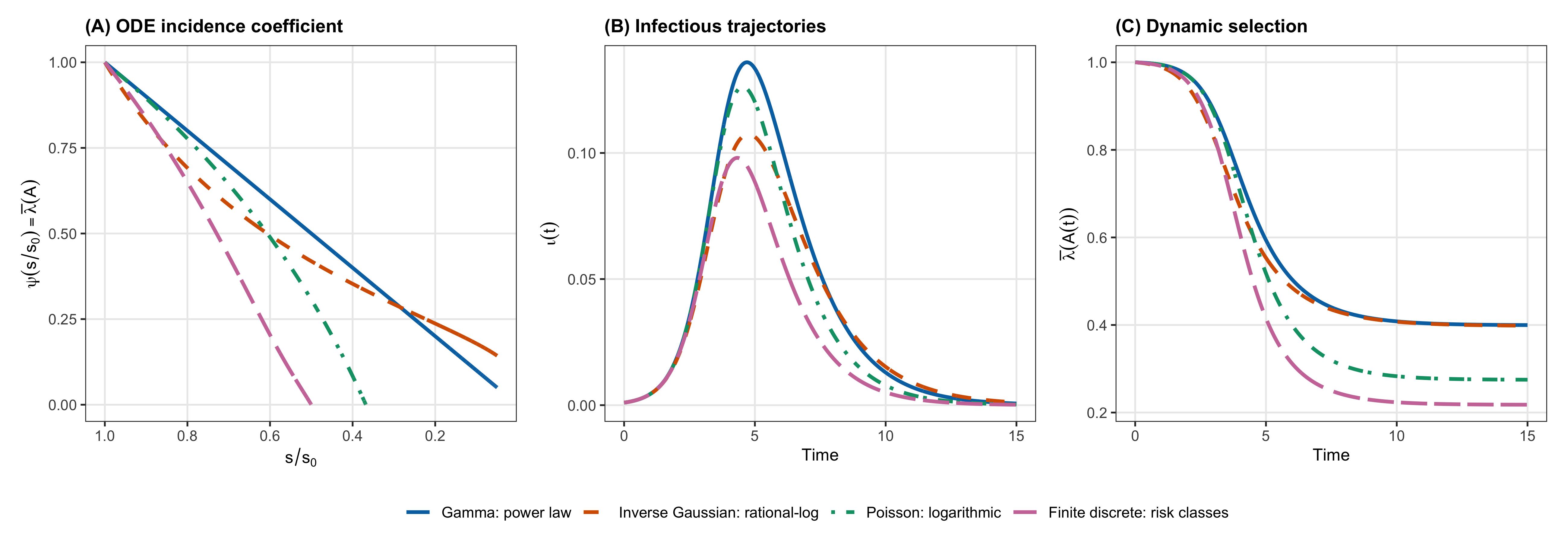}
    \caption{
Comparison of epidemic ODEs induced by four susceptibility distributions, each normalized to have
$\mathbb{E}[\Lambda]=1$ and simulated with the same recovery rate, initial conditions, and initial reproduction number. The following parameters have been used in this figure: $\alpha=1, \kappa=1$ for Gamma distribution, $m=1, \nu=0.5$ for Inverse-Gaussian distribution, $\theta=1$ for Poisson distribution, and $\Pp(\Lambda=0)=0.5, \Pp(\Lambda=1)=0.25$, and $\Pp(\Lambda=3)=0.25$ for the finite discrete distribution. 
\textbf{(A)} The effective susceptibility
$\psi(s/s_0)=\overline{\lambda}(A)$ entering the incidence term
$\tau\psi(s/s_0)s\iota$.
The Gamma distribution yields the benchmark power-law closure, whereas the inverse-Gaussian, Poisson, and finite discrete distributions produce qualitatively different non-power-law incidence functions.
\textbf{(B)} The corresponding infectious trajectories.
Although the models agree initially, dynamic selection generates differences in peak prevalence, peak timing, and epidemic duration.
\textbf{(C)} The mean susceptibility
$\overline{\lambda}(A(t))$ among individuals who remain susceptible.
Its decline reflects the preferential removal of highly susceptible individuals; the limiting value is positive when a resistant or low-risk subgroup remains.
    }
    \label{fig:susceptibility-ode-comparison}
\end{figure}
%\subsection{Comparison of the five examples}

These examples show that the general mixed-Sellke construction produces a broad family of closed epidemic ODEs, rather than only a power-law modification of the homogeneous SIR model. Figure~\ref{fig:susceptibility-ode-comparison} illustrates this point by comparing four representative closures under a common initial mean susceptibility and initial reproduction number. Gamma susceptibility gives the benchmark power-law incidence coefficient and connects directly with the classical heterogeneous-SIR literature. The inverse-Gaussian law provides a continuous non-Gamma example with a rational--logarithmic closure, while Poisson susceptibility produces a logarithmic incidence coefficient together with a resistant atom. A finite discrete law permits arbitrary multimodality and yields an exact class-structured system. As seen in Figure~\ref{fig:susceptibility-ode-comparison}, these different closure functions generate distinct rates of dynamic selection and, consequently, different infectious peak sizes, peak times, and epidemic durations, even though the models agree at the start of the epidemic. The negative-binomial example, not displayed in the figure, further extends the discrete family by allowing overdispersed susceptibility and a difference-of-powers incidence coefficient.

\subsection{Connections with heterogeneous models and applications}

The five examples are analytically closed representatives of broader
modeling traditions; not every exact reduced incidence function appears to
have been used previously under the same parametrization.

The Gamma case has the clearest epidemiological precedent. Novozhilov showed
that Gamma-distributed susceptibility produces a power transmission function
in a heterogeneous SIR model \cite{Novozhilov2008}, and \cite{DiekmannInaba2023}
placed this reduction in a general method for incorporating separable static
heterogeneity into compartmental epidemic models.
Gamma-distributed susceptibility or exposure was also used in analyses of
selection and herd-immunity thresholds for SARS--CoV--2
\citep{GomesEtAl2022,MontalbanEtAl2022}. Thus, the Gamma mixed-Sellke model
provides, within the present framework, an individual-level stochastic foundation
for a nonlinear ODE already used
in COVID--19 and general epidemic modeling.

Finite discrete laws correspond to classical multigroup or risk-class models,
with classes representing different levels of susceptibility associated, for
example, with age, prior immunity, vaccination status, or genetic type.
Related discrete and continuous susceptibility classes have been used to study
prevalence and susceptibility variance in SIS models
\citep{MargheriEtAl2015}. Models based on discrete contact classes and
superspreading have also motivated reduced SIR formulations for COVID--19
\citep{Szapudi2020}, although in such models heterogeneity may act through
contact activity rather than susceptibility alone.

The Poisson and negative-binomial cases are natural when susceptibility is a
count of approximately additive exposure opportunities, contacts, receptors,
or risk factors. Their atom at zero has the interpretation of a completely resistant or immune fraction. Compound-Poisson frailty models were developed precisely to
combine unobserved heterogeneity with a nonsusceptible fraction and include
Gamma and inverse-Gaussian frailty models in a broader family
\citep{WienkeEtAl2010}. The negative-binomial law is especially useful when the
latent count is overdispersed relative to Poisson.

The inverse-Gaussian example connects directly with frailty models, where both
Gamma and inverse-Gaussian mixing laws are standard choices for persistent
unobserved heterogeneity. Their Laplace transforms convert individual
proportional hazards into marginal population hazards and generate the same
selection phenomenon as in Proposition~\ref{prop:selection}
\citep{Aalen1994,BalanPutter2019}.

Beyond epidemics, individual-level distributional heterogeneity is important
in microbial growth. Alonso, Molina, and Theodoropoulos used fitted Gamma
distributions for single-cell division characteristics to connect stochastic
single-cell dynamics with bacterial population growth
\citep{AlonsoEtAl2014}. Although their population equations are not the same as
the SIR reduction here, the modeling principle is parallel: a tractable
individual mixing law yields a low-dimensional population description while
retaining variability lost by replacing all individuals with a single mean.

\section{Summary and Discussion}\label{sec:discussion}

We have shown that a broad class of heterogeneous deterministic SIR models
admits a common stochastic interpretation: these models arise as
law-of-large-numbers limits of finite stochastic epidemics in which individuals
carry fixed susceptibility marks and infection is generated through a
generalized Sellke construction. In the large-population limit, the
resulting epidemic dynamics are determined by the Laplace transform of the
initial susceptibility distribution. This provides an individual-level
probabilistic foundation for a broad class of deterministic heterogeneous
epidemic models that have traditionally been formulated directly at the
population level
\citep{Novozhilov2008,Katriel2009,DiekmannInaba2023}.

One consequence of this construction is a transparent interpretation of
dynamic selection. The nonlinear incidence coefficient is simply the average
susceptibility of those individuals who remain susceptible after the
accumulated infection pressure. Rather than replacing the heterogeneous
population by a single effective transmission rate, the limiting equations
retain the changing composition of the susceptible population throughout the
epidemic. In this way, the selective depletion of highly susceptible
individuals emerges naturally from the underlying stochastic model, providing
a probabilistic counterpart to mechanisms that have long been recognized in
epidemiology and survival analysis
\citep{VaupelYashin1985,GomesEtAl2022,Aalen1994}.

The examples demonstrate that this construction is considerably more general
than the classical Gamma model. Besides recovering the familiar power-law
incidence, they identify several other susceptibility distributions that also
lead to finite-dimensional deterministic systems. This suggests that
tractable heterogeneous compartmental models are closely connected with the
analytic properties of Laplace transforms, rather than with any particular
choice of susceptibility distribution. From this perspective, the search for
closed epidemic models becomes a question of identifying distributions whose
Laplace transforms satisfy suitable differential or algebraic relations.

There are several natural directions in which the present framework may be
extended. We have considered heterogeneity only in susceptibility. Allowing
heterogeneity in infectiousness, contact activity, or infectious-period
duration would require joint marks and could alter both the selection
mechanism and the generation interval. Likewise, although the present paper
focuses on the SIR model, the same construction should extend naturally to
models with additional disease stages, such as SEIR systems. In such an
extension, the susceptibility distribution among individuals who remain
susceptible would evolve according to the same exponential-tilting mechanism
described here, while the exposed compartment would introduce an additional
transition between infection and infectiousness. Related survival-based
reductions for SEIR models have been studied in \cite{KhudaBukhsh2024HowTo},
although a full mixed-Sellke fluid-limit treatment of the heterogeneous SEIR
case is beyond the scope of the present paper. This is consistent with
existing deterministic formulations of static heterogeneity
given in \cite{Novozhilov2008,DiekmannInaba2023,DiekmannHeesterbeekBritton2013}. By
contrast, extending the approach to SIS-type models appears substantially more
difficult. Since recovered individuals return to the susceptible class, the
susceptibility distribution is no longer determined solely by cumulative
infection pressure, and the exponential-tilting representation developed here
does not apply directly. A successful extension would therefore require
tracking the susceptibility distribution across individuals with different
infection histories.

The present work also highlights a close connection between heterogeneous
epidemic models and classical frailty theory. Mathematically, the exponential
tilting of the susceptibility distribution is identical to the selection
mechanism underlying frailty models in survival analysis
\citep{VaupelYashin1985,Aalen1994,WienkeEtAl2010}. The Laplace transform
therefore plays the same unifying role in epidemic dynamics as it does in
frailty theory, linking persistent individual heterogeneity with
population-level behavior.

The deterministic limits derived here may also be useful as a basis for parameter estimation, for instance, via the so-called   
Dynamical Survival Analysis (DSA). The central idea of DSA is to interpret
solutions of population-level epidemic equations in terms of the event-time
distributions of individuals sampled from the population. In particular, the
susceptible trajectory can be viewed as determining the survival distribution
of an initially susceptible individual, while the remaining compartmental
trajectories provide corresponding descriptions of infection, recovery, and
other disease-transition times. This perspective makes it possible to move
from a mechanistic ODE model to likelihood-based inference for individual
event-time data, including incomplete and right-censored observations
\citep{KhudaBukhshEtAl2020,DiLauroEtAl2022,RempalaKhudaBukhsh2023}.

The heterogeneous ODE systems obtained in the present paper broaden this
approach by incorporating persistent individual variation directly into the
population dynamics. Their solutions determine marginal individual survival
functions after averaging over the susceptibility distribution, while also
retaining conditional survival descriptions for individuals with specified
susceptibility marks. The resulting framework could therefore be used to
estimate epidemic parameters jointly with parameters describing population
heterogeneity, compare alternative susceptibility distributions, and assess
how unobserved heterogeneity changes inferred infection risks and epidemic
forecasts. More generally, it illustrates how deterministic limits of
interacting stochastic systems can provide tractable approximations to
individual event-time likelihoods. A related use of this principle in
stochastic reaction networks was recently developed by Ganguly and
KhudaBukhsh, who derived likelihood-based inference procedures from
individual product-formation times using interacting-particle approximations
and propagation-of-chaos arguments
\citep{GangulyKhudaBukhsh2026}.

The principal message of this paper is that heterogeneous compartmental
epidemic models need not be viewed as purely phenomenological deterministic
systems. Instead, many of the heterogeneous models studied in the literature
arise naturally as law-of-large-numbers limits of finite stochastic epidemics
with fixed individual susceptibility marks. The generalized Sellke
construction therefore provides a common stochastic foundation for these
models while also suggesting a systematic route to constructing new
tractable deterministic epidemic models from probabilistic assumptions on
individual susceptibility.

\section*{Statements and Declarations}

\bmhead{Funding}
Olga Izyumtseva was supported by the British Academy through grant number RaR{\textbackslash}100741, and in part by British Academy, Cara, Leverhulme Trust through grant LTRSF24{\textbackslash}100014 and LTRSF26{\textbackslash}100038. Greg Rempala was partially supported by the  HELM Initiative at The Ohio State University.  

% The authors declare that no funds, grants, or other support were received during the preparation of this manuscript.

\bmhead{Competing Interests}
The authors have no relevant financial or non-financial interests to disclose.

\bmhead{Author Contributions}
Author contributions will be specified in the final submitted version.

\bmhead{Data Availability}
No datasets were generated or analyzed for the present theoretical study.

\bibliography{refs}

\appendix

\section{Proof of the general fluid-limit theorem}
\label{app:general-proof}

The proof has three ingredients. First, a Glivenko--Cantelli argument controls
the empirical resistance distribution uniformly over all pressure levels.
Second, a functional law of large numbers controls the recovery process.
Finally, these two estimates reduce the epidemic to an integral equation
with a uniformly vanishing perturbation, to which Gr\"onwall's inequality is
applied.

For the \(S_N(0)\) initially susceptible individuals, define the empirical
survival function
\[
G_N(a)
=
\frac{1}{S_N(0)}
\sum_{j=1}^{S_N(0)}
\ind_{\{Q_j>a\}},
\qquad a\geq 0.
\]

Since \(s_N(0)\xlongrightarrow{\Pp}s_0>0\), we have
\[
S_N(0)\xlongrightarrow{\Pp}\infty.
\]
By the Glivenko--Cantelli theorem with random sample size,
\[
\sup_{a\geq 0}|G_N(a)-G(a)|
\xlongrightarrow{\Pp}0.
\]

The Sellke construction gives the exact identity
\[
s_N(t)=s_N(0)G_N(A_N(t)).
\]
Consequently, for every fixed \(T>0\),
\[
\sup_{0\leq t\leq T}
\left|
s_N(t)-s_0G(A_N(t))
\right|
\xlongrightarrow{\Pp}0.
\]

Next, define
\[
M_N(t)
=
\frac{1}{N}
\left[
Y_R\left(
N\gamma\int_0^t\iota_N(u)\,du
\right)
-
N\gamma\int_0^t\iota_N(u)\,du
\right].
\]
The functional law of large numbers for the Poisson process gives
\[
\sup_{0\leq t\leq T}|M_N(t)|
\xlongrightarrow{\Pp}0.
\]
Therefore
\[
r_N(t)
=
r_N(0)+
\gamma\int_0^t\iota_N(u)\,du+M_N(t).
\]
Since
\[
A_N(t)=\tau\int_0^t\iota_N(u)\,du,
\]
we obtain
\[
r_N(t)
=
r_N(0)+\frac{\gamma}{\tau}A_N(t)+M_N(t).
\]

Using population conservation,
\[
\iota_N(t)=1-s_N(t)-r_N(t),
\]
we find that \(A_N\) satisfies a perturbed integral equation
\[
A_N(t)
=
\int_0^t F(A_N(u))\,du+\eta_N(t),
\]
where
\[
F(a)
=
\tau[1-r_0-s_0G(a)]-\gamma a
\]
and
\[
\eta_N(t)
=
\tau\int_0^t
\Bigl[
(r_0-r_N(0))
+
\bigl(
s_0G(A_N(u))
-
s_N(0)G_N(A_N(u))
\bigr)
-
M_N(u)
\Bigr]\,du .
\]
Then, for every fixed \(T>0\),
\begin{equation*}
\sup_{0\le t\le T}|\eta_N(t)| 
\le
\tau T
\Biggl(
|r_N(0)-r_0|
+
\sup_{0\le u\le T}
\left|
s_N(0)G_N(A_N(u))
-
s_0G(A_N(u))
\right|
+
\sup_{0\leq t\leq T}|M_N(t)|
\Biggr).
\end{equation*}

Moreover,
\[
\begin{aligned}
\sup_{0\le u\le T}
\left|
s_N(0)G_N(A_N(u))
-
s_0G(A_N(u))
\right|
&\le
|s_N(0)-s_0|
+
\sup_{a\ge0}|G_N(a)-G(a)|.
\end{aligned}
\]
and, consequently, 
\[
\sup_{0\le t\le T}|\eta_N(t)| 
\le
\tau T
\left(
|r_N(0)-r_0|
+
|s_N(0)-s_0|
+
\sup_{a\ge0}|G_N(a)-G(a)|
+
\sup_{0\leq t\leq T}|M_N(t)|
\right).
\]
Each term on the right-hand side converges to zero in probability:
the first two by convergence of the initial proportions, the third by the
Glivenko--Cantelli theorem with random sample size, and the last by the
functional law of large numbers for the recovery Poisson process. Hence
\[
\sup_{0\leq t\leq T}|\eta_N(t)|
\xlongrightarrow{\Pp}0.
\]
Because \(G\) is globally Lipschitz,
\[
|F(a)-F(b)|
\leq
\left(\tau s_0\E[\Lambda]+\gamma\right)|a-b|.
\]
Thus, \(F\) is globally Lipschitz. Let \(A\) be the unique solution of
\[
A(t)=\int_0^tF(A(u))\,du.
\]
Then
\[
|A_N(t)-A(t)|
\leq
\sup_{0\leq t\leq T}|\eta_N(t)|
+
L\int_0^t|A_N(u)-A(u)|\,du,
\]
where
\[
L=\tau s_0\E[\Lambda]+\gamma.
\]
By Gr\"onwall's inequality,
\[
\sup_{0\leq t\leq T}|A_N-A|
\leq
e^{LT}\sup_{0\leq t\leq T}|\eta_N(t)|
\xlongrightarrow{\Pp}0.
\]

The convergence of the susceptible process follows from
\[
s_N(t)=s_0G(A_N(t))+o_{\Pp}(1)
\]
uniformly on \([0,T]\) and the Lipschitz continuity of \(G\). The convergence of
the removed process follows from
\[
r_N(t)
=
r_N(0)+\frac{\gamma}{\tau}A_N(t)+M_N(t),
\]
and the infectious process converges because
\(
\iota_N(t)=1-s_N(t)-r_N(t).
\)
This completes the proof.

\end{document}